\documentclass[12pt,leqno]{article}

\usepackage{amsmath,amssymb,amsthm,color}
\usepackage[all]{xy}
\usepackage{lscape}

\makeatletter

\newtheorem{theorem}{Theorem}[section]
\newtheorem{proposition}[theorem]{Proposition}
\newtheorem{lemma}[theorem]{Lemma}
\newtheorem{corollary}[theorem]{Corollary}

\theoremstyle{definition}

\newtheorem{example}[theorem]{Example}
\newtheorem{definition}[theorem]{Definition}

\theoremstyle{remark}

\theoremstyle{remark}
\newtheorem{remark}[theorem]{Remark}

\def\({{\rm (}}
\def\){{\rm )}}

\let\Mathrm\operator@font

\newcommand{\bn}{\ensuremath{\mathbb N}}

\def\standop#1{\mathop{\Mathrm #1}\nolimits}
\def\difstop#1#2{\expandafter\def\csname #1\endcsname{\standop{#2}}}
\def\defstop#1{\difstop{#1}{#1}}

\defstop{AB}
\defstop{ann}
\defstop{Ass}

\defstop{CMFI}
\defstop{codim}
\defstop{Coh}
\defstop{coht}
\defstop{Coker}
\defstop{Cone}
\defstop{Cl}
\defstop{Cox}
\defstop{cosk}

\defstop{depth}
\defstop{Div}
\defstop{div}

\defstop{EM}
\defstop{embdim}
\defstop{End}
\defstop{ev}
\defstop{Ext}

\defstop{Flat}
\defstop{Func}
\defstop{Fpqc}

\defstop{Gal}

\defstop{GD}
\defstop{Good}

\difstop{height}{ht}
\defstop{Hom}

\def\id{\mathord{\Mathrm{id}}}

\difstop{Image}{Im}
\defstop{Inv}
\defstop{ind}

\defstop{Ker}

\defstop{Lch}
\defstop{length}
\defstop{Lin}
\defstop{Lqc}
\defstop{lqc}
\defstop{LQ}
\defstop{LM}
\defstop{Lie}

\defstop{Mat}
\defstop{Max}
\defstop{Min}
\defstop{Mod}
\defstop{Mor}
\defstop{MCM}
\defstop{Map}

\defstop{Nerve}
\defstop{NonSFR}
\defstop{NonCMFI}
\defstop{NonNor}
\defstop{Nor}

\defstop{ob}

\defstop{PA}
\defstop{PM}
\defstop{Proj}
\defstop{Prin}
\defstop{Pic}

\defstop{Ps}

\defstop{Qch}
\defstop{qch}

\defstop{rad}
\defstop{rank}

\defstop{res}
\defstop{Reg}
\defstop{Ref}

\defstop{Sh}

\defstop{Spec}
\defstop{supp}
\defstop{Supp}
\defstop{Sym}
\defstop{Sing}
\defstop{SFR}
\defstop{Soc}

\defstop{Ob}

\difstop{tdeg}{trans.deg}

\defstop{Tor}
\difstop{trace}{tr}

\defstop{Zar}
\defstop{Ker}
\defstop{Arf}

\def\NN{\mathbb{N}}
\def\ZZ{\mathbb{Z}}

\def\sdarrow#1{\downarrow\hbox to 0pt{\scriptsize$#1$\hss}}
\def\suarrow#1{\uparrow\hbox to 0pt{\scriptsize$#1$\hss}}
\def\ssearrow#1{\searrow\hbox to 0pt{\scriptsize$#1$\hss}}

\def \dim{\operatorname{dim}}

\def \rank{\operatorname{rank}}
\def \coloneq{\mathrel{\mathop:}=}
\def \id{\operatorname{Id}}

\def \set{\operatorname{Set}}
\def \emdim{\text{em.dim}}

\def\section{\@startsection{section}{1}{\z@ }%
  {-3.5ex plus -1ex minus -.2ex}{2.3ex plus .2ex}{\bf }}

\long\def\refname{\par\kern -3ex
  \begin{center}\rm R\sc{eferences}\end{center}\par\kern 
  -2ex}

\def\@seccntformat#1{\csname the#1\endcsname.\quad}

\def\@@@sect#1#2#3#4#5#6[#7]#8{%
  \ifnum #2>\c@secnumdepth 
  \def \@svsec {}\else \refstepcounter {#1}%
  \def\@svsec{}
  \fi 
  \@tempskipa #5\relax 
  \ifdim \@tempskipa >\z@ 
  \begingroup #6\relax \@hangfrom {\hskip #3\relax 
    \@svsec}{\interlinepenalty \@M #8\par }\endgroup 
  \csname #1mark\endcsname {#7}
  \else 
  \def \@svsechd {#6\hskip #3\@svsec #8\csname #1mark\endcsname {#7}}
  \fi \@xsect {#5}}

\def\@@@startsection#1#2#3#4#5#6{%
  \if@noskipsec \leavevmode \fi \par \@tempskipa #4\relax \@afterindenttrue 
  \ifdim \@tempskipa <\z@ \@tempskipa -\@tempskipa \@afterindentfalse 
  \fi \if@nobreak \everypar {}\else \addpenalty {\@secpenalty }\addvspace 
  {\@tempskipa }\fi \@ifstar {\@ssect {#3}{#4}{#5}{#6}}{\@dblarg 
    {\@@@sect {#1}{#2}{#3}{#4}{#5}{#6}}}}

\def\theparagraph{\thesection.\arabic{paragraph}}
\def\aparagraph{\@@@startsection{paragraph}{2}{\z@ }%
  {1.75ex plus .2ex minus .15ex}{-1em}{\bf(\theparagraph) } }
\def\paragraph{\@@@startsection{paragraph}{2}{\z@ }%
  {1.75ex plus .2ex minus .15ex}{-1em}{}{\bf(\theparagraph)} }

\let\c@theorem\c@paragraph

\begin{document}

\begin{center}  
{\LARGE A new proof of non-Cohen-Macaulayness of Bertin's example} \\
~\\
{Takuma Seno} 
\end{center}

\begin{abstract}
Bertin's example is famous as the first known Noetherian UFD that is not Cohen-Macaulay.
In the example,
she employed a ring of invariants and proved that the ring is not Cohen-Macaulay by calculating a homogeneous system of parameter and generators of it.
In this paper,
we give a new proof by arguments on ring theoretic properties.

\end{abstract}

\section{Introduction}

Pierre Samuel asked his student,
Marie-Jos\'e Bertin,
if every UFD is Cohen-Macaulay or not.
She answered this question negatively by studying an example in which a cyclic group of order
$4$
acting on
$K[x_1 , x_2 , x_3 , x_4]$
by permutating the variables.
Larry Smith proved that a ring of invariants of
$3$-dimensional representation of a finite group is always Cohen-Macaulay.
So the representation in Bertin's example has the minimum dimension that the ring of invariants is not Cohen-Macaulay.
Her proof of non-Cohen-Macaulayness depends on calculations of a homogeneous system of parameter (h.s.o.p.\ for short) and generators of it.
The main subject of this paper is giving a new proof of non-Cohen-Macaulayness.
We also introduce a theorem which is a generalization  of the proof.

Throughout this paper,
let
$G$
be a finite group,
$K$
be a field,
$V$
be a finite dimensional representation of
$G$,
and
$\NN$
be the set of nonnegative integer.
(That is,
$0 \in \NN$.)

In Section
$2$,
we introduce the definitions of the Hilbert series of an
$\NN$-graded ring,
and some properties of it.
We also introduce Stanley's result.
He proved that the Gorensteinness of some kind of an
$\NN$-graded ring depends only on its Hilbert series.
In Section 3,
we give a description of invariant theory of a finite group.
Invariant theory of a finite group is classified into two cases.
One is called the modular case and the other is called the nonmodular case.
In the nonmodular case,
a ring of invariants is always Cohen-Macaulay and
there is a well-known characterizations of Gorensteinness,
which is called Watanabe's theorem.
Furthermore,
we can calculate the Hilbert series of a ring of invariants by Molien's theorem.
However,
in the modular case,
the situation is complicated.
A ring of invariants is not always Cohen-Macaulay,
and the above theorems doesn't hold in general.
Amiram Braun proved that if a ring of invariants is Cohen-Macaulay,
Watanabe'theorem is true in the modular case.
We introduce the definition of a permutation representation and related theorems.
When
$V$
is a permutation representation,
every homogeneous part of the ring of invariants is generated by the orbit sums of all monomials.
M.\ G\"{o}bel obtained a good result about generators as
$K$-algebra.
By using G\"{o}bel's theorem,
we can prove that the ring of invariants of the
$n$-dimensional alternating group
$A_n$
is a hypersurface.
Section 4 is the main part of this paper.
In this section,
we briefly review Bertin's original proof of non-Cohen-Macaulayness of the ring of invariants and then give a new proof of it.
We refer the reader to Ellingsrud and Skjelbred \cite{Elli Skje} and Kemper \cite{Kemp}
for results on depths and (non-)Cohen-Macaulayness of the ring of invariants under the action of finite group,
including different proofs of Bertin's example treated in this paper.
We also describe a generalization of our new method of proof.

Acknowledgment.
The author is grateful to Professor Gregor Kemper for valuable advice.

\section{From commutative algebra}

\begin {definition} \label{def Hilb seri}
Let
$R$
be a positively graded finitely generated
$K$-algebra with
$R_0 = K$.
\[
H ( R , \lambda ) \coloneq \sum_{n \in \bn} \dim_K R_n \lambda^n
\]
We call the power series the Hilbert series of $R$.
Let $x_1 , \dots , x_d$ be an h.s.o.p.\ of
$R$.
Then,
the Hilbert series of $R$ is represented as follows. 
\[ H ( R , \lambda ) = \frac{h(\lambda)}{ (1-\lambda^{a_1}) (1-\lambda^{a_2}) \cdots (1-\lambda^{a_d})} \ \ \ \ \ h(\lambda) \in \ZZ[\lambda] \ , 
\]
where $d=\dim R , \ a_i = \deg x_i$.
\end{definition}

If
$R$
is Cohen-Macaulay,
the following holds.

\begin{proposition}\label{formula Hilb diveded}
Let
$R$ be an \bn-graded Cohen-Macaulay ring with $R_0 = K$,
and
$x_1 , \dots x_d$
be an h.s.o.p.\ of
$R$.
Then,
\[
H(R/(x_1 , \dots x_i) , \lambda) =H(R , \lambda) \prod_{k=1}^i(1-\lambda^{a_k})
\]
for
$i=1 , \dots d$,
where
$a_i = \deg x_i$.
\end{proposition}

This proposition says that if
$R$
is Cohen-Macaulay,
$h(\lambda)$,
the numerator of the Hilbert series,
corresponds to the number and the degrees of the free generators of
$R$
as a
$K[x_1 , \dots , x_d]$
module.
That is,
let
$y_1 , \dots , y_r$
be the free generators and
$c_i = \deg y_i$
then,
\[
h(\lambda) = \sum^r_{i=1} \lambda^{c_i}
\]

\begin{definition}
Let
$R$
be a Noetherian commutative ring graded by
$\NN$.
We say
$R$
is a G-algebra if
$R_0 = K$
is satisfied.

\end{definition}

\begin{theorem}\label{Stan}
\({\bf Stanley}\)
(\cite{Stan},Theorem 4.4)
Let
$R$
be a G-algebra.
Suppose that
$R$
is a Cohen-Macaulay integral domain of Krull dimension d.
Then
$R$
is Gorenstein if and only if for some
$\rho \in \ZZ$,
\[
H \left(R,\frac{1}{\lambda} \right)=(-1)^d\lambda^\rho H(R,\lambda) .
\]
\end{theorem}

The condition in Stanley's theorem can be rephrased  that the numerator of the Hilbert series is \lq \lq palindromic.\rq \rq

\begin{definition}
We say that a polynomial $f(x)$ is palindromic if there exists an integer $n$ such that $x^nf(\frac{1}{x}) = f(x)$.
\end{definition}

This definition is equivalent to say that $a_k = a_{k+r} , a_{k+1} = a_{k+r-1} , \dots $
for
$f(x) = a_{k}x^k + \cdots + a_{k+r}x^{k+r} \ (a_k , a_{k+r} \neq 0)$.

\begin{proposition}
Let $R$ be a Cohen-Macaulay ring .
Then,
$h(\lambda)$
is a palindromic polynomial if and only if the Hilbert series of
$R$
satisfies the conclusion of Stanley's theorem.
\($h(\lambda)$ is defined in Definition \ref{def Hilb seri}.\)
\end{proposition}

\begin{proof}
If
$h(\lambda)$
is a palindromic polynomial,
\[
H\left( R , \frac{1}{\lambda} \right) = \frac{h\left( \frac{1}{\lambda} \right)}{\left( 1- \left( \frac{1}{\lambda} \right)^{a_1} \right) \cdots \left( 1-  \left( \frac{1}{\lambda} \right)^{a_d} \right)} \ \ \ \ \ \ \ \ \ \ \ \ \ \ \ \ \ \ \ 
\]
\[
\ \ \ \ \ \ \ \ \ \ \ \ \ \ \ \ \ 
= \frac{\lambda^l \lambda^m h(\lambda)}{(-1)^d (1-\lambda^{a_1}) \cdots (1-\lambda^{a_d})} \ \ \ \left(l:=\sum^d_{i=1}a_i\right)
\]
\[ \ \ \ \ \ \ \ \ \ \ \ \ \ \ 
= \frac{\lambda^\rho h(\lambda)}{(-1)^d (1-\lambda^{a_1}) \cdots (1-\lambda^{a_d})}  \ \ \ \ (\rho = l+m)
\]
\[
= (-1)^d \lambda^\rho H(R , \lambda) . \ \ \ \ \ \ \ \ \ \ \ \ \ \ \ \ \ \ \ \ 
\]
For the converse,
\[
h\left(\frac{1}{\lambda} \right) = H\left(R , \frac{1}{\lambda}\right) \left( 1- \left( \frac{1}{\lambda} \right)^{a_1} \right) \cdots \left( 1-  \left( \frac{1}{\lambda} \right)^{a_d}  \right) 
\]
\[
\ \ \ \ \ \ \ \ \ 
=(-1)^d\lambda^\rho H(R,\lambda) (-1)^d (1-\lambda^{a_1}) \cdots (1-\lambda^{a_d})
\]
\smallskip
\[
= \lambda^\rho h(\lambda). \ \ \ \ \ \ \ \ \ \ \ \ \ \ \ \ \ \ \ \ \ \ \ \ \ \ \ \ \ \ \ \ \ \ \ \ \ \ 
\]
\end{proof}

\begin{remark}
It is not so difficult to see that
$h \( \lambda \)$
is palindromic if
$R$
is a Gorenstein ring.
Let
$x_1 , \dots x_d$
be an h.s.o.p.\ of
$R$.
By considering the quotient ring
$R/(x_1 , \dots , x_d)$,
it comes down to the case that
$R$
is Artinian.

In this statement,
we don't need the condition that
$R$
is \lq \lq domain.\rq \rq
What is great in Stanley's result is to have found out a sufficient condition for the converse.
We introduce the outline of his result as follows.
Let
$A := k[Y_1, \dots , Y_s]$
be a polynomial ring and
$y_1 , \dots , y_s$
be a homogeneous generators of
$R$.
That is,
$R = k[y_1 , \dots , y_s]$.
$R$
is Cohen-Macaulay so we can take a finite free resolution of
$R$
as an
$A$-module.
\[ 0 \rightarrow M_h \rightarrow \cdots \rightarrow M_0 \rightarrow R \rightarrow 0 : exact \]
Set
$\( - \)^* := \Hom_A \( - , A \) ,\ K_R :=\Ext^{s-d}_A \( R,A\) $.
Then,
$M_i \simeq M^*_i$.
With some degree shift,
$K_R$
coincides the canonical module of
$R$.
From above exact sequence,
we obtain
\[
 0 \rightarrow M^*_0 \rightarrow \dots \rightarrow M^*_h \rightarrow K_R \rightarrow 0 :exact 
.\]
Hilbert series of
$R$
(similarly of
$K_R$
)
is calculated as the alternating sum of Hilbert series of
$M_i$.
\[
 H \( M_i , \lambda \) = \frac{\sum^{\beta \( i \) }_{j=1}\lambda^{g_{ij}}}{\Pi^s_{t=1} \( 1- \lambda^{e^t} \) } 
\]
where 
$X_{i1} , \dots , X_{i \beta\( i \)}$
is a basis of
$M_i$,
$g_{ij} := \deg X_{ij} $
The conclusion follows from palindromicness of a numerator of Hilbert series of
$R$.
\end{remark}

\section{From invariant theory}

\begin{theorem}
Let
$S_n$
be the symmetric group of degree
$n$.
Then,
\[
K[V]^{S_n} = K[s_1 , s_2 , \dots , s_n], \ 
where \ \ 
s_i = \sum_{1 \le k_1 < k_2 < \dots < k_i \le n} x_{k_1} \cdots x_{k_i}
\]
Each
$s_i$
is algebraic independent so
$K[s_1 , s_2 , \dots , s_n]$
is polynomial ring.
\end{theorem}

If
$G \subset S_n$,
the following corollary is immediate
(we should pay attention to that
$\dim K[V] = \dim K[V]^G$).
\begin{corollary}\label{elem symm sop}
Let
$V$
be a
$n$-dimensional permutation representation of
$G$.
Then,
$s_1 , \dots , s_n$
is an h.s.o.p.\ of
$K[V]^G$.

\end{corollary}

Invariant theory of a finite group is classified into the modular case and the nonmodular case.
In the nonmodular case,
Hochster and Eagon proved that the ring of invariants is always a Cohen-Macaulay ring,
and K.Watanabe got a comprehensible characterization of Gorensteiness.

\begin{definition}\label{pseudo}
Let
$G$
be a finite group and
$V$
be a representation of
$G$.
We say that
$g \in G$
is a pseudoreflection
if
$\rank (\id - g ) =1$
satisfied.
\end{definition}

\begin{theorem}\label{wtnb}
\({\bf K.Watanabe}\)
Let
$\( | G | , p \) = 1$.
Then,
$K[V]^G$
is a Gorenstein ring if
$G \subset SL \( V \)$.
The converse holds if
$G$
contains no pseudoreflection.
\end{theorem}

\begin{theorem}\label{Moli}
\({\bf Molien}\)
Let
$K$
be a field of characteristic
$0$.
Then,
\[
H \left(R,\lambda \right) = \frac{1}{|G|} \sum_{\sigma \in G} \left( \frac{1}{\det ( \operatorname{Id} -  \lambda \sigma)} \right).
\]
\end{theorem}

On the other hand,
in the modular case,
the situation is complicated.
We cannot say that the ring of invariants is Cohen-Macaulay in general.
Generalizations of Theorem \ref{wtnb} called Watanabe type theorem was actively studied.
Amiram Braun proved the generalization to the modular case. (Peter Fleischmann--Chris Woodcock also proved some result independently and almost simultaneously.)
We introduce Braun's result here.

\begin{definition}
Let
$g$
be a pseudoreflection.
We say
$g$
is a transvection if
$g$
is not diagonalizable.
If it is diagonalizable,
it is said to be a homology.

\end{definition}

\begin{remark}
In the nonmodular case,
every pseudoreflection is a homology.
In fact,
if
$g$
is a transvection,
its Jordan normal form is represented as follows.
\[
\left(
\begin{array}[]{cc|ccc} 
  1 & 1 & {} & {} & {} \\
  0 & 1 & {} & {} & {}  \\ \hline
  {} & {} & 1 & {} & {} \\
  {} & {} & {}& \ddots & {}\\ 
  {} & {} & {} & {} & 1 \\ 
\end{array} 
\right)
\]
This matrix has order
$p$.
If there exists any transvection in the nonmodular case,
it contradicts that
$(p , |G|) = 1$.
\end{remark}

\begin{theorem}\label{Bra1}
\({\bf Braun}\)
(\cite{Brau},Theorem B)
Let
$G \subset SL(V)$
be a finite group which contains no transvections.
Then,
the Cohen-Macaulay locus of
$S(V)^G$
coincides with its Gorenstein locus.
In particular,
if
$S(V)^G$
is Cohen-Macaulay then it is also Gorenstein.
\end{theorem}

\begin{theorem}\label{Bra2}
\({\bf Braun}\)
(\cite{Brau},Theorem C)
Suppose that
$G \subset GL(V)$
is a finite group with no pseudoreflection(of any type)and
$S(V)^G$
is Gorenstein.
Then,
$G \subset SL(V)$.
\end{theorem}

In Bertin's example,
$G$
acts on
$V$
as a permutation representation.
In this case,
$K[V]^G$
is generated by the orbit sums of all monomials.
And therefore,
the Hilbert series of
$K[V]^G$
is independent of the characteristic of
$K$.

\begin{definition}\label{permu repre}
Let
$\{ e_1 , \dots , e_n \}$
be a basis of
$V$
and
$\{x_1, \dots x_n\}$
be the dual basis of
$V^*$
with respect to
$\{ e_1 , \dots , e_n \}$.
We say that
$V$
is a permutation representation of
$G$
if for any
$g \in G$
and any
$i$,
there exists
$j$
such that
$g\(e_i\) = e_j$
is satisfied.
This is equivalent to say that for any
$g \in G$
and any
$i$,
there exists
$j$
such that
$g\(x_i\) = x_j$.
\end{definition}

\begin{theorem}\label{gene by orbit sum}
Let
$V$
be a permutation representation of
$G$.
Then,
$K[V]^G$ is generated over
$K$
by the orbit sums of all monomials.
\end{theorem}

\begin{proof}
For any
$g \in G$,
the action of
$g$
on
$K[V]^G$
is degree preserving.
So,
it is sufficient to prove that every homogeneous part
$K[V]^G_{\(n\)}$
is generated by the orbit sums of all monomials of
$K[V]^G_{\(n\)}$.
Let
$f \in K[V]^G_{\(n\)}$
We can write
$f$
as follows.
\[
f = \sum_{\deg (I) = d} a_I x^I \ (a_I \in K )
\]
For any $g \in G$,
$g(f)=\sum_{\deg (I) = d} a_I \cdot g(x^I)$.
So,
$a_I = a_J$
if
$x^J \in Gx^I$.
Let
$\mathcal{O}_G\(x^I\)$
denote orbit sum of
$x^I$.
Then,
$f-a_I\mathcal{O}_G\(x^I\) \in K[V]^G_{\(n\)}$.
We can finish proof by induction on the number of monomials contained in
$f$.
\end{proof}

\begin{corollary}\label{Hilb seri independent}
Let
$V$
be a permutation representation of
$G$.
Then,
the Hilbert series of
$K[V]^G$
is independent of the characteristic of
$K$.
\end{corollary}

\begin{definition}\label{gap}
For
$A \in \NN^n$,
we let
$\set (A)$
denote
$\{a_1 , \dots a_n\}$,
$\operatorname{ht}(A) = \max \{a_i \mid i=1 , \dots n \}$,
where
$A = (a_1 , \dots , a_n)$
We say
$x^A$
has a gap (at
$r$) if there exists a number
$r \in \NN$
such that
$\{ r+1 , \dots , \operatorname{ht}(A) \} \subset \set (A)$,
and
$r \not \in \set(A)$.

\end{definition}

\begin{theorem}\label{Gobe}
\({\bf G\"{o}bel}\)
Let
$V$
be a permutation representation of
$G$.
Then,
\[
\{  \mathcal{O}_G(x^A) \mid x^A \ does \ not \ have \ a \ gap \} \cup \{x_1x_2 \cdots x_n\}
\]
is a generating set for
$K[V]^G$.

\end{theorem}

By applying G\"{o}bel's theorem,
we can prove that a ring of invariants of
$A_n$
is a hypersurface.

\begin{definition}
We say that a Noetherian ring
$R$
is a hypersurface if
$\emdim R \le \dim R + 1$
is satisfied.

\end{definition}

\begin{theorem}
$K[V]^{A_n}$
is a hypersurface.

\end{theorem}

\begin{proof}
By Theorem \ref{Gobe},
$K[V]^{A_n}$
is generated by the orbit sums which have no gap.
Fix
$I_0 = (0 , 1 , \dots , n-1)$.
We show that if
$x^I \not \in A_nx^{I_0}$,
$\mathcal{O}_G(x^I) \in K[V]^{S_n}$.
If
$x^I \not \in A_nx^{I_0}$,
there exists
$i,j$
such that
$a_i = a_j$,
$i \neq j$.
So we obtain
$(i \ j )x^I = x^I$.
For any
$\sigma \in S_n \backslash A_n$,
there exists
$\tau_\sigma \in A_n$
such that
$\sigma = \tau_\sigma(i \ j)$.
Hence,
\[
\sigma \mathcal{O}_{A_n}(x^I) = \tau_\sigma \sum_{\tau \in A_n} (i \ j) \tau x^I = \tau_\sigma \sum_{\tau \in A_n} \tau (i \ j) x^I 
\]
\[
\ \ \ \ \ \ \ \ \ \ \ \ \ \ \ \ \ = \tau_\sigma \sum_{\tau \in A_n} \tau x^I = \tau_\sigma \mathcal{O}_{A_n}(x^I) = \mathcal{O}_{A_n}(x^I)
\]
(We should pay attention to that
$\sigma A_n = A_n \sigma$
for all
$\sigma \in S_n$.)
Therefore,
$\mathcal{O}_{A_n}(x^I) \in K[V]^{S_n}$,
and
\[
K[V]^{A_n} = K[V]^{S_n}[\mathcal{O}_{A_n}(x^{I_0})] = K[s_1 , \dots , s_n][\mathcal{O}_{A_n}(x^{I_0})]
\]
\end{proof}

\section{Bertin's example}
In this section,
we give a new proof of Bertin's celebrated example of a ring of invariants that is not Cohen-Macaulay (this is the main part of this paper).
The ring is the first example of a UFD which is not Cohen-Macaulay.
For the new proof,
we prepare a lemma.

\begin{lemma}\label{pseudo transposition}
Let
$G$
be a permutation group.
Then,
$g \in G$
is a pseudoreflection if and only if
$g$
is a transposition.
\end{lemma}

\begin{proof}
Whether
$g$
is a pseudoreflection or not is stable under conjugation.
So we permutate a basis of
$V$
if we need to do.
Any permutation can be represented as a product of some cyclic permutations.
A cyclic permutation of length
$r$
is represented by a matrix conjugate to
\\ 
{\color{white}--------------------------------}
$\overbrace{\ \ \ \ \ \ \ \ \ \ \ \ \ \ \ \ \ \ }^{r}$
 \vspace{-4mm}
\[
\left(
\begin{array}[]{cccc|ccc} 
  0 & {} & {} & 1\\
  1 & \ddots & {} & {} & {} & \text{\huge{0}} \\
  {} & \ddots & \ddots & {}\\ 
  {} & {} & 1 & 0 \\ \hline
  {} & {} & {} & {} & 1 \\
  {} & \text{\huge{0}} & {} & {} & {} & \ddots \\
  {} & {} & {} & {} & {} & {} & 1
\end{array} 
\right).
\]

If
$g \in G$
is a product of
$2$
or more cyclic permutations then,
$\rank \id - g \ge 2$.
And,
\\
{\color{white}-------------------------------------}
$\overbrace{\ \ \ \ \ \ \ \ \ \ \ \ \ \ \ \ \ \ \ \ \ \ \ }^{r}$
 \vspace{-4mm}
\[
\left(
\begin{array}{cccc} 
  1 & {} & {} & -1\\
  -1 & \ddots & {} & {}  \\
  {} & \ddots & \ddots & {}\\ 
  {} & {} & -1 & 1 \\ 
\end{array} 
\right).
\]
has rank
$r-1$.
Therefore,
if
$g$
is a pseudoreflection,
\[
g =
\left(
\begin{array}{cc|ccc} 
  0 & 1 & 0 & \dots & 0 \\
  1 & 0 & 0 & \dots & 0 \\ \hline
  0 & 0 & 1 & \\ 
  \vdots & \vdots & {} & \ddots  \\
  0 & 0 & {}  & {}&1
\end{array} 
\right).
\]
This is transposition.
For the converse,
if
$g \in G$
is a transposition,
of course,
$g$
is a cyclic permutation of length
$2$.
So
$\rank \id - g = 1$.
Hence
$g$
is a pseudoreflection.
\end{proof}

\begin{example}
Let
$G$
be a subset of the symmetric group of degree 4,
generated by
$\sigma = \(1 \ 2 \ 3 \ 4 \)$.
$G$
acts on
$K[x_1,x_2,x_3,x_4]$
by permutation of variables.
That is,
$G$ acts on
$\{ x_1,x_2,x_3,x_4 \}$
with
\[
\sigma \(x_i\) = x_{\sigma\(i\)}
\]
Then,
$K[V]^G$
is not Cohen-Macaulay.
\end{example}

Before introducing the new proof,
we describe a draft of Bertin's original proof.
By Corollary \ref{Hilb seri independent},
$K[V]^G$,
the Hilbert series does not depend on characteristic $p$.
To calculate Hilbert function,
we can assume $p=0$ and apply Theorem \ref{Moli}.
\[
H ( K[V]^G , \lambda ) =\frac{1}{4} \left( \frac{1}{(1- \lambda)^4} + \frac{1}{(1- \lambda ^2)^2} + \frac{2}{1- \lambda^4} \right)
\]
\[
\ \ \ \ \ \ \ \ \ \ \ \ =
\frac{1 + \lambda^2 + \lambda^3 + 2 \lambda^4 + \lambda^5}{(1-\lambda)(1-\lambda^2)(1-\lambda^3)(1-\lambda^4)}
\]
By Proposition \ref{formula Hilb diveded}\ 
and Corollary \ref{elem symm sop},
if
$K[V]^G$
is Cohen-Macaulay,
it has an h.s.o.p.\ 
$s_1 , \dots , s_n$
and free generators
$f_1 , \dots , f_6$
on
$K[s_1 , \dots , s_n]$
with degree
$1 , 3 , 4 , 4 , 5$.
But there are no free generators which satisfy this condition.
Hence
$K[V]^G$
is not Cohen-Macaulay.

In Bertin's original proof,
we need large amount of calculation for searching relations of generators.
In our new proof,
we argue Gorensteinness of
$K[V]^G$
instead of calculation of generators.

Our new proof is as follows.
\begin{proof} 
We assume
$K[V]^G$
is Cohen Macaulay for
$p=2$.
We already saw the Hilbert series and its numerator is not palindromic.
By theorem \ref{Stan},
$K[V]^G$
is not Gorenstein for
$p=2$.
(If
$p \neq 2$,
it's the nonmodular case.
So
$K[V]^G$
is not Gorenstein for any characteristic
$p$.)
On the other hand,
by Lemma \ref{pseudo transposition},
$G$
contains no pseudoreflection.
Because
$V$
is a permutation representation of
$G$,
$\det (g) = \pm 1$
for all
$g \in G$.
If
$p=2$,
$1 = -1$.
Hence
$G \subset SL(V) $.
So,
by Theorem \ref{Bra1},
$K[V]^G$
is Gorenstein.
This is a contradiction.
Thus,
$K[V]^G$
is not Cohen-Macaulay ring for
$p=2$.
\end{proof}

We can generalize this proof as follows.

\begin{theorem}
Let
$G$
be a subset of the symmetric group acting on 
$K[V]$
by permutations of variables and contains an odd permutation but not contains any pseudoreflection.
Then,
the ring of invariants is not Cohen-Macaulay for
$p=2$.
\end{theorem}

\begin{proof}
We assume that
$K[V]^G$
is Cohen-Macaulay for
$p=2$.
$G$
contains odd permutation so,
if
$p=0$,
$G \not \subset SL(V)$.
By Lemma \ref{pseudo transposition},
$G$
contains no pseudoreflection.
Therefore,
by Theorem \ref{wtnb},
$K[V]^G$
is not Gorenstein so a numerator of Hilbert series is not palindromic.
On the other hand,
if
$p=2$,
$G \subset SL(V)$.
By Theorem \ref{Bra1},
$K[V]^G$
is Gorenstein so a numerator of Hilbert series is palindromic.
This is contradiction.
Thus,
$K[V]^G$
is not Cohen-Macaulay for
$p=2$.
\end{proof}





Seno Takuma \\
Department of Mathmatics \\
Osaka Metropolitan University \\
Sumiyoshi-ku Ssaka 558-8585, JAPAN \\
e-mail: {\tt takumasenoo71@gmail.com}


\begin{thebibliography}{9}
\bibitem{Brau}
A. Braun, 
On the Gorenstein property for modular invariants,
{\em J.Algebra},
{\bf 345} (2011),
81–99.

\bibitem{Elli Skje}
G. Ellingsrud and T. Skjelbred,
Profondeur d'anneaux d'invariants en caract\'eristique p,
{\em Compositio Math},
{\bf 41} (1980),
233-244.

\bibitem{Kemp}
G. Kemper, 
On the Cohen-Macaulay Property of Modular Invariant Rings,
{\em J.Algebra},
{\bf 215} (1999),
330–351.

\bibitem{Camp}
H. E. A. Eddy Campbell and David L. Wehlau,
Modular Invariant Theory,
Springer, 2011.

\bibitem{Flei Wood}
P. Fleischmann and C. Woodcock, Relative invariants, ideal classes and quasicanonical modules of modular rings of invariants,
{\em J. Algebra},
{\bf 348} (2011),
110–134.

\bibitem{Stan}
R. P. Stanley,
Hilbert functions of Graded Algebra,
{\em Adv. in Math},
{\bf 28} (1978)
57–83.

\bibitem{Brun Herz}
W. Bruns and J. Herzog, Cohen-Macaulay Rings, Cambridge University Press, 1993.



\end{thebibliography}
\end{document}